\documentclass[10pt]{amsart}
\pdfoutput=1
\usepackage{latexsym, amssymb, amsmath, graphicx}

\theoremstyle{plain}
\newtheorem{thm}{Theorem}
\newtheorem{cor}{Corollary}

\newtheorem{lem}{Lemma}

\begin{document}
\title[Combinatorial methods in right-angled Artin groups]
{Combinatorial methods for detecting surface subgroups in right-angled
Artin groups}
\author{Robert W. Bell}
\address{Lyman Briggs College\\
	W-32 Holmes Hall\\
	Michigan State University\\
	East Lansing, MI 48825-1107}
\email{rbell@math.msu.edu}
\date{November 28, 2010}
\begin{abstract}
	We give a short proof of the following theorem of Sang-hyun Kim:  if 
	$A(\Gamma)$ is a right-angled Artin group with defining graph 
	$\Gamma$, then $A(\Gamma)$ contains a hyperbolic surface subgroup if 
	$\Gamma$ contains an induced subgraph $\overline{C}_n$ for some 
	$n \geq 5$, where $\overline{C}_n$ denotes the complement graph of an 
	$n$-cycle.  Furthermore, we give a new proof of Kim's co-contraction
	theorem. 
\end{abstract}

\maketitle

\section{Introduction and Definitions}

Suppose $\Gamma$ is a simple finite graph with vertex set $V_\Gamma$ 
and edge set $E_\Gamma$.  We say that $\Gamma$ is the defining graph of
the right-angled Artin group defined by the presentation
\[A(\Gamma) = \langle V_\Gamma \ ; \ [v,w]:=vwv^{-1}w^{-1}=1 
\ \forall \ \{v,w\} \in E_\Gamma \rangle.\]
Right-angled Artin groups are also called graph groups or  partially 
commutative groups in the literature. All graphs in this article
are assumed to be simple and finite.

Right-angled Artin groups have been studied using both combinatorial and 
geometric methods.  In particular, it is well-known that these groups have 
simple solutions to the the word and conjugacy problems. Moreover,
each right-angled Artin group can be geometrically represented as the 
fundamental group of a nonpositively curved cubical complex $X_\Gamma$
called the Salvetti complex.  For these and other fundamental results, we 
refer the reader to the survey article by Charney \cite{CharneyMR2322545}.  

Let $\Gamma$ be a graph, and suppose that $W \subset V_\Gamma$.  The induced 
subgraph $\Gamma_W$ is the maximal subgraph of $\Gamma$ on the vertex set 
$W$.  A subgraph $\Lambda \subset \Gamma$ is called an induced subgraph if 
$\Lambda = \Gamma_{V_\Lambda}$.  In this case, the subgroup of $A(\Gamma)$ 
generated by $V_\Lambda$ is canonically isomorphic to $A(\Lambda)$.  This 
follows from the fact that $f:A(\Gamma) \to A(\Lambda)$ given by $f(v) = v$
if $v \in V_\Lambda$ and $f(v) = 1$ if $v \notin V_\Lambda$ defines
a retraction.  Therefore, we identify $A(\Lambda)$ with its image 
in $A(\Gamma)$.

In this article, we study the following problem: find conditions on
a graph $\Gamma$ which imply or deny the existence of hyperbolic surface
subgroup in $A(\Gamma)$.  Herein, we say that a group is a hyperbolic surface
group if it is the fundamental group of a closed orientable surface with
negative Euler characteristic. 

Droms, Servatius, and Servatius proved that if $\Gamma$ contains
an induced $n$-cycle, i.e. the underlying graph of a regular $n$-gon, 
for some $n \geq 5$, then $A(\Gamma)$ has a hyperbolic surface subgroup 
\cite{ServatiusDromsServatiusMR952322}.  In fact, they construct
an isometrically embedded closed surface of genus $1 + (n-4)2^{n-3}$
in the cover of the Salvetti complex $X_\Gamma$ corresponding to the 
commutator subgroup of $A(\Gamma)$. 

Kim and, independenly, Crisp, Sageev, and Sapir gave the first examples of 
graphs without induced $n$-cycles, $n \geq 5$, which define a right-angled 
Artin groups which, nonetheless, contain hyperbolic surface subgroups 
\cite{KimMR2443098, CrispSageevSapirMR2422070}.  We give one such example 
here to illustrate the main lemma of this article.

Consider the graphs in Figure~\ref{doublegraph}. The map 
$\phi: A(P') \to A(P)$ sending $v_1$ to $v_1^2$, $v_i$ to $v_i$, and 
$v_i'$ to $v_i$ for $i > 1$ defines an injective homomorphism onto an index 
two subgroup of $A(P)$ (see the discussion below).  Since $P'$ contains an 
induced circuit of length five (the induced subgraph on the vertices 
$v_2, \dots, v_5$ and $v_6'$), $A(P)$ contains a hyperbolic surface subgroup; 
however $P$ does not contain an $n$-cycle for any $n \geq 5$.

\begin{figure} \label{doublegraph}
\includegraphics[scale=.75]{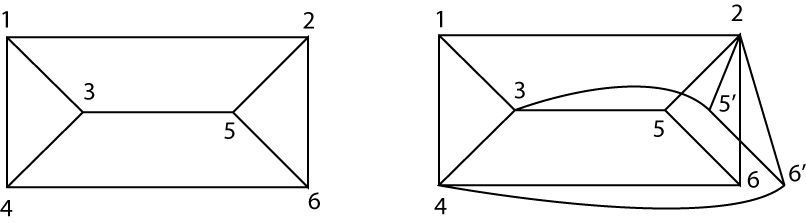}
\caption{The group given by the graph $P$ on the right injects into
the group given by the graph $P'$ on the left.  The vertex labeled by $i$
is referred to as $v_i$ in the discussion below.}
\end{figure}

That the map $\phi$ is injective can be seen from several persepectives.
The approach of Kim and Crisp, Sageev, and Sapir is to use dissection 
diagrams; these are collections of simple closed curves which are dual to 
van Kampen diagrams on a surface over the presentation $A(\Gamma)$. The
method was introduced in this context by Crisp and Wiest and used with much 
success by Kim and Crisp, Sageev, and Sapir 
\cite{CrispWiestMR2077673, KimMR2443098, CrispSageevSapirMR2422070}.

The purpose of this article is show how classical methods from combinatorial 
group theory can offer a somewhat different perspective and to simplify some of 
the arguments.  We will use the Reidemeister-Schreier rewriting process to
give a direct proof that the map $\phi$, above, is injective; and we also 
indicate how this can be proven using normal forms for splittings of groups.
This in turn will lead to a short proof of Kim's theorem on 
co-contractions (Theorem 4.2 in \cite{KimMR2443098}) alluded to in the
abstract; see Theorem~\ref{Kimtheorem} in this article.

\begin{lem}\label{main}
Suppose $A(\Gamma)$ is a graph group, and let $n$ be a positive 
integer.  Choose  a vertex $z \in V(\Gamma)$, and define 
$\phi:A(\Gamma) \to \langle x \ ; \ x^n = 1 \rangle 
\cong \mathbb{Z}/n\mathbb{Z}$ by 
$\phi(v) = 1$ if $v \neq z$ and $\phi(z) = x$.  
Then $\ker{\phi}$ is a graph group whose defining graph $\Gamma'$ is 
obtained by gluing $n$ copies of 
$\Gamma \backslash \, st(z)$ to $st(z)$ along $lk(z)$, where
$st$ and $lk$ are the star and link, respectively.  Moreover,
the vertices of $\Gamma'$ naturally correspond to the following generating
set: 
\[\{z^2\} \cup lk(z) \cup \{u : u \notin st(z)\} \cup \{zuz^{-1}: u \notin
st(z)\} \cup \cdots \cup \{z^{n-1}uz^{1-n}: \notin st(z)\}.\]
\end{lem}

The proof is a fairly straightforward computation (or geometric observation 
from the point of view of covering space theory) using the Reidemeister-
Schreier method.  The details are given in Section \ref{RSM}.  
Applying Lemma~\ref{main} to the graphs in Figure~\ref{doublegraph}, proves 
that $A(P')$ injects into $A(P)$: if $\phi: A(P) \to \mathbb{Z}/2\mathbb{Z}$ 
maps $z = v_1$ to 1 mod 2, then $A(P') = \ker \phi$.

Another way to prove Lemma~\ref{main} is to take advantage of ``visual''  
splittings of the groups $A(\Gamma)$ and $A(\Gamma')$ as an HNN extension
or amalgamated free product.  This second approach is stated in the 
article by Crisp, Sageev, and Sapir (see Remark 4.1 in 
\cite{CrispSageevSapirMR2422070}).  

We illustrate the utility of Lemma~\ref{main} by giving a short proof of the 
following theorem of Kim.

\begin{thm}[Kim \cite{KimMR2443098}, Corollary 4.3 (2)]\label{Kim}
	Let $\overline{C}_n$ denote the complement graph of an $n$-cycle.
	For each $n \geq 5$, the group $A(\overline{C}_n)$ contains a 
	hyperbolic surface subgroup.  
\end{thm}

In fact, we give a new proof of Kim's more general theorem
(Theorem 4.2 in \cite{KimMR2443098}) on co-contractions 
of right-angled Artin groups in Section \ref{twotheorems}.  Kim's proof used 
the method of dissection diagrams. Kim has also discovered a short proof 
using visual splittings (personal communication).

In preparing this article, we found that the Reidemeister-Schreier method has 
been used previously to study certain Bestvina-Brady subgroups of right-angled 
Artin groups (see \cite{LevyParkerVanWykMR1626045} and 
\cite{BeckerHorakVanWykMR1611316}).

This work was inspired by a desire to better understand Crisp, 
Sageev, and Sapir's very interesting classification of the graphs on fewer than
nine vertices which define right-angled Artin groups with hyperbolic surface
subgroups.  We hope that our retelling of this small piece of theirs and
Kim's work will help to clarify some aspects of the general problem.

The author would like to thank several colleagues for advice and feedback:
Jon Hall, Ian Leary, Ulrich Meierfrankenfeld, and especially Mike Davis and 
Tadeusz Januskiewicz.  Furthermore, the author owes many of the ideas
herein to Kim, Crisp, Sageev, and Sapir, as this article is essentially 
an account of some of his attempts to better understand their body of work on 
this problem.

\section{The Reidemeister-Schreier Method and proof of Lemma~\ref{main}} 
\label{RSM}

The Reidemeister-Schreier method solves the following problem: suppose that 
$G$ is a group given by the presentation $\langle X \,;\, R \rangle$, and 
suppose $H \subset G$ is a subgroup; find a presentation for $H$.  
The treatment below is brisk; see \cite{BaumslagMR1243634} 
for details and complete proofs.

Let $F = F(X)$ be free with basis $X$, and let $\pi: F \to G$ extend the 
identity map on $X$.  Consider the preimage $P = \pi^{-1}(H)$.  
Let $T \subset F$ be a right Schreier transversal for $P$ in $F$, i.e. $T$ 
is a complete set of right coset representatives that is closed under the 
operation of taking intial subwords (of freely reduced words over $X$).  
Given $w \in F$, let $[w]$ be the unique element of $T$ such that $Pw = P[w]$. 
For each $t \in T$ and $x \in X$, let $s(t,x) = tx[tx]^{-1}$.  Define 
$S = \{ s(t,x) : t \in T, x \in X, \text{ and } s(t,x) \neq 1\}$.
Then $S$ is a basis for the free group $P$.  Define a
rewriting process $\tau: F \to P$ on freely reduced words over $X$ by 
\[\tau(y_1 y_2 \cdots y_n) = s(1,y_1) s([y_1],y_2) \cdots 
s([y_1 \cdots y_{n-1}],y_n),\]
where $y \in X \cup X^{-1}$.  Then $\tau(w) = w[w]^{-1}$ for 
every reduced word $w \in F$, and 
\[H = \langle S \ ; \ 
\tau(t^{-1}rt) = 1 \ \forall \ t \in T, r \in R \rangle.\]
This rewriting process together with the resulting presentation for the given 
subgroup $H$ of $G = \langle X; R \rangle$ is called the Reidemeister-Schreier
Method.

\emph{Proof of Lemma~\ref{main}.} 
Let $\Gamma$ be a graph, $G = A(\Gamma)$ the 
corresponding right-angled Artin group, and $z$ a distinguished vertex of 
$\Gamma$.  Let $\phi: G \to \langle x\, ;\, x^n \rangle$ be given by 
$\phi(v) = 1$ if $v \neq z$ and $\phi(z) = x$.  

Let $F$ be free on $X = V_{\Gamma}$ and let $R$ be
the set of defining relators corresponding to $E_{\Gamma}$.
Let $H = \ker{\phi}$, and let $P$ be the inverse image of $H$ in $F$ under 
the natural map $F \to G$. The set $T = \{\, 1, z, \dots, z^{n-1} \,\}$ 
is a right Schreier transversal for $P < F$.  One verifies (directly) 
that the following equations hold: 

\[
\begin{array}{ll}
s(z^k, v) = z^k v z^{-k}& 
\text{ if } v \neq z \text{ and } k = 0, \dots, n-1\\

s(z^k ,z) = 1&
\text{ if } k = 0, \dots, n-2\\

s(z^k, z) = z^n& 
\text{ if } k = n-1.  
\end{array}
\]

Thus, we have a set $S$ of generators for $\ker \phi$; however, many of 
these generators are redundant.  Again, one verifies (using
$\tau(w) = w[w]^{-1}$) that the following equations hold:

\[
\begin{array}{ll}
\tau(z^k[u,v]z^{-k}) = [(s(z^k,u),s(z^k,v)]& 
\text{ if } u, v \neq z \text{ and } k= 0, \dots, n-1\\

\tau(z^k [z,v] z^{-k}) = s(z^{k+1},v) \cdot s(z^{k},v)^{-1}&
\text{ if } v \neq z \text{ and } k = 0, \dots n-2\\

\tau(z^k [z,v] z^{-k}) = z^n \cdot s(1,v) \cdot z^{-n} 
\cdot s(z^{n-1},v)^{-1}&
\text{ if } v \neq z \text{ and } k = n-1.  
\end{array}
\]

Therefore, if $[z,v] = 1$ is a relation in $A(\Gamma)$, then 
\[v = s(1,v) = s(z,v) = \dots = s(z^{n-1},v) \text{ and } [z^n,v] = 1\]
hold in $\ker \phi$.  It follows that $\ker \phi$ is generated by
$z^n$, the vertices adjacent to $z$ in $\Gamma$, and $n$ copies
($u, zuz^{-1}, \dots, z^{n-1}uz^{1-n}$) of each vertex $u \neq z$ 
and not adjacent to $z$ in $\Gamma$.  Moreover, the relations are such
that $\ker \phi$ is presented as a right-angled Artin group where the 
defining graph is obtained from $\Gamma$ by taking the star of $z$ and $n$ 
copies of the complement of the star of $z$ and gluing these copies along the 
link of $z$.  This completes the proof of Lemma~\ref{main}.

\section{A short proof of two theorems of Kim} \label{twotheorems}

Suppose that $\Gamma$ is a graph.  The complement graph $\overline{\Gamma}$
is the graph having the same vertices as $\Gamma$ but which has edges 
complementary to the edges of $\Gamma$. Recall that an $n$-cycle $C_n$ is
the underlying graph of a regular $n$-gon.

Theorem~\ref{Kim} follows from Kim's co-contraction theorem 
(see Theorem~\ref{Kimtheorem} below); however, we present a short independent 
proof here.

\emph{Proof of Theorem~\ref{Kim}.}
Suppose that $\Gamma$ is a graph which contains an induced $C_5$.  Then 
$A(\Gamma)$ contains a hyperbolic surface subgroup by 
\cite{ServatiusDromsServatiusMR952322}.  Since $C_5 \cong \overline{C}_5$, 
Theorem~\ref{Kim} follows from the following Lemma~\ref{cn}.

\begin{lem}[Kim \cite{KimMR2443098}, Corollary 4.3 (1)]\label{cn}
For each $n \geq 4$, $A(\overline{C}_{n-1}) < A(\overline{C}_{n})$.
\end{lem}

\begin{proof} Let $V_{C_n} = \{ x_1, \dots, x_n\} = V_{\overline{C}_n}$.  
Define $\phi: A(\overline{C}_{n}) \to \langle a \ ; \ a^2 \rangle$ by 
$\phi(x_n) = a$ and $\phi(x_i) = 1$ for $i \neq n$.  By Lemma~\ref{main}, the 
defining graph $\Gamma$ of $\ker{\phi}$ has vertex set 
$V_{\Gamma} = \{z^2\} \cup \{x_1, \dots, x_{n-1}\} \cup \{ y_1, y_{n-1} \}$, 
where $z = x_n$ and $y_i = z x_i z^{-1}$.
Let $S = \{y_1, x_2, \dots, x_{n-1}\}$.  Consider the induced subgraph 
$\Gamma_S$.  The vertices $y_1$ and $x_i$ are not adjacent if and only if 
$i \in \{2, n-1\}$.  The vertices $x_i$ and $x_j$ are not adjacent if and 
only if $|i-j| \leq 1$.  Therefore, $\Gamma_{S} \cong \overline{C}_{n-1}$.
\end{proof}

Kim proved a more general theorem about subgroups of a right-angled Artin 
group $A(\Gamma)$ defined by ``co-contractions''.  Let $S \subset V_{\Gamma}$, 
and let $S' = V_{\Gamma} \backslash S$.  If $\Gamma_S$ is 
connected, then the contraction $CO(\Gamma,S)$ of $\Gamma$ relative
to $S$ is defined by taking the induced subgraph $\Gamma_{S'}$ together 
with a vertex $v_{S}$ and declaring $v_{S}$ to be adjacent to 
$w \in S'$ if $w$ is adjacent in $\Gamma$ to some vertex in $S$.
The co-contraction $\overline{CO}(\Gamma,S)$ is defined as follows:
\[\overline{CO}(\Gamma,S) = \overline{CO(\overline{\Gamma},S)}\]

Kim insists that $\Gamma_S$ be connected whenever he considers the contraction
$CO(\Gamma,S)$.  This assumption is not necessary.  Moreover, the following
lemma shows that the structure of $\Gamma_S$ is immaterial; the proof
follows directly from the definitions.

\begin{lem} \label{Sindependence}
Suppose $\Gamma$ is a graph and $S \subset V_\Gamma$. Let $\Gamma'$ be
the graph obtained from $\Gamma$ by removing any edges joining two elements
of $S$.  Then $CO(\Gamma,S) = CO(\Gamma',S)$.
\end{lem}

\begin{cor}
Suppose $\Gamma$ is a graph and $S \subset V_\Gamma$.  Let $\Gamma'$ be
a graph obtained from $\Gamma$ by adding or deleting any collection of
edges with both of their vertices belonging to $S$.  Then $CO(\Gamma,S)
= CO(\Gamma',S)$ and $\overline{CO}(\Gamma,S) = \overline{CO}(\Gamma',S)$.
\end{cor}

\begin{lem} \label{twosuffice}
Suppose $\Gamma$ is a graph and $n \geq 2$.  Suppose $S = \{s_1, \dots, s_n\}
\subset V_{\Gamma}$.  
Let $\Lambda = \overline{CO}(\Gamma,\{s_1, \dots, s_{n-1})$ and 
$S' = S \backslash \{s_n\}$.  Then
$\overline{CO}(\Gamma,S) = \overline{CO}(\Lambda,\{v_{S'},s_{n}\})$.
\end{lem}

\begin{proof}
It suffices to compare the collection of vertices which are adjacent to 
$v_{S}$ in $\Gamma_1=\overline{CO}(\Gamma,S)$  and 
$\Gamma_2=\overline{CO}(\Lambda,\{v_{S'},s_{n}\})$; in the latter case, we are
identifying the vertex $v_{S' \cup \{s_n\}}$ with $v_{S}$.

A vertex $w$ in $\Gamma_1$ not belonging to $S$ is adjacent to $v_S$ if and only
if $w$ is adjacent to every $s_i$ in $\Gamma$.  A vertex $w$ in $\Gamma_2$ not 
equal to $v_{S'}$ nor $s_n$ is adjacent to $v_S$ if and only if $w$ is adjacent
to $v_{S'}$ and $s_n$ in $\Lambda$; but this, in turn, means that $w$ is 
adjacent to every $s_i$ in $\Gamma$.  (Note that the case of $n = 2$ is 
trivial since $CO(\Gamma,\{s_1\}) = \Gamma)$.)
\end{proof}

A collection of vertices $S \subset V_{\Gamma}$ is said to be anti-connected
if $\overline{\Gamma}_{S}$ is connected.  (Note: $(\overline{\Gamma})_{S} =
\overline{(\Gamma_S)}$.)

\begin{thm}[Kim \cite{KimMR2443098}, Theorem 4.2] \label{Kimtheorem}
Suppose $\Gamma$ is a graph and $S \subset V(\Gamma)$ is an anti-connected 
subset.  Then $A(\overline{CO}(\Gamma,S))$ embeds in $A(\Gamma)$.
\end{thm}

\begin{proof}
First consider the case when $S$ consists of two non-adjacent vertices 
$z,z' \in V_{\Gamma}$.
Define $\phi: A(\Gamma) \to \langle x \ ;\ x^2 \rangle$  by $\phi(z) = x$,
and $\phi(v) = 1$ if $v \neq z$.  Let $A(\Gamma') = \ker \phi$.  Let
\[T = (V(\Gamma) \backslash \{z^2,z'\}) \cup \{zz'z^{-1}\} 
\subset V(\Gamma').\]
We claim that $\Gamma'_{T} \cong \overline{CO}(\Gamma,S)$ via 
$v \mapsto v$ if $v \neq zz'z^{-1}$ and $zz'z^{-1} \mapsto v_{S}$.

If $v$ and $w$ are distinct from $z$ and $z'$, then $v$ and
$w$ are adjacent in $\overline{CO}(\Gamma,S)$ if and only if
they are adjacent in $\Gamma$.

On the other hand, a vertex $w$ is adjacent to $v_S$ in 
$\overline{CO}(\Gamma,S)$ if and only if $w$ is adjacent to $z$ and $z'$,  
whereas a vertex $w$ is adjacent to $zz'z^{-1}$ in $\Gamma'_{T}$ if and only 
if $w$ belongs to the link of $z$ and to the link of $z'$, i.e. $w$ is adjacent
to $z$ and $z'$.  Therefore, $\Gamma'_{T} \cong \overline{CO}(\Gamma,S)$
and, hence, $A(\overline{CO}(\Gamma,S))$ embeds in $A(\Gamma)$.  

Now we prove the general statement by induction on $|S|$.  Suppose 
$S = \{s_1, \dots, s_n\}$ is anti-connected, and suppose we have chosen the 
ordering so that $S' =\{s_1, \dots, s_{n-1}\}$ is also anti-connected.  (This 
is always possible: choose $s_n$ so that it is not a cut point of 
$\overline{\Gamma}_{S}$.) Let $\Lambda = \overline{CO}(\Gamma,S')$.  Suppose 
that $A(\Lambda)$ embeds in $A(\Gamma)$.  By the case of two vertices above, 
$A(\overline{CO}(\Lambda,\{v_{S'}, s_n\}))$ embeds in $A(\Lambda)$.  (Note 
that $v_{S'}$ and $s_n$ are not adjacent in $\Lambda$ for, otherwise, $s_n$
would be adjacent to every $s_i$, $i = 1, \dots n-1$, which would contradict
the hypothesis that $S$ is anti-connected.)  This proves the inductive step.
The proof of the theorem is completed by applying Lemma~\ref{twosuffice}.

\end{proof}

\bibliography{mybib}{}

\begin{thebibliography}{1}

\bibitem{BaumslagMR1243634}
Gilbert Baumslag.
\newblock {\em Topics in combinatorial group theory}.
\newblock Lectures in Mathematics ETH Z\"urich. Birkh\"auser Verlag, Basel,
  1993.

\bibitem{BeckerHorakVanWykMR1611316}
Jennifer Becker, Matthew Horak, and Leonard VanWyk.
\newblock Presentations of subgroups of {A}rtin groups.
\newblock {\em Missouri J. Math. Sci.}, 10(1):3--14, 1998.

\bibitem{CharneyMR2322545}
Ruth Charney.
\newblock An introduction to right-angled {A}rtin groups.
\newblock {\em Geom. Dedicata}, 125:141--158, 2007.

\bibitem{CrispSageevSapirMR2422070}
John Crisp, Michah Sageev, and Mark Sapir.
\newblock Surface subgroups of right-angled {A}rtin groups.
\newblock {\em Internat. J. Algebra Comput.}, 18(3):443--491, 2008.

\bibitem{CrispWiestMR2077673}
John Crisp and Bert Wiest.
\newblock Embeddings of graph braid and surface groups in right-angled {A}rtin
  groups and braid groups.
\newblock {\em Algebr. Geom. Topol.}, 4:439--472, 2004.

\bibitem{KimMR2443098}
Sang-hyun Kim.
\newblock Co-contractions of graphs and right-angled {A}rtin groups.
\newblock {\em Algebr. Geom. Topol.}, 8(2):849--868, 2008.

\bibitem{LevyParkerVanWykMR1626045}
Joshua Levy, Cameron Parker, and Leonard Van~Wyk.
\newblock Finite presentations of subgroups of graph groups.
\newblock {\em Missouri J. Math. Sci.}, 10(2):70--82, 1998.

\bibitem{ServatiusDromsServatiusMR952322}
Herman Servatius, Carl Droms, and Brigitte Servatius.
\newblock Surface subgroups of graph groups.
\newblock {\em Proc. Amer. Math. Soc.}, 106(3):573--578, 1989.

\end{thebibliography}
\bibliographystyle{plain}
\end{document}